\documentclass[12pt]{amsart}  

\usepackage{custom_preamble} 

\begin{document}

\title{Approximate groups and doubling metrics}

\author{\tsname}
\address{\tsaddress}
\email{\tsemail}

\begin{abstract}
We develop a version of Fre{\u\i}man's theorem for a class of non-abelian groups, which includes finite nilpotent, supersolvable and solvable $A$-groups.  To do this we have to replace the small doubling hypothesis with a stronger relative polynomial growth hypothesis akin to that in Gromov's theorem (although with an effective range), and the structures we find are balls in (left and right) translation invariant pseudo-metrics with certain well behaved growth estimates.

Our work complements three other recent approaches to developing non-abelian versions of Fre{\u\i}man's theorem by Breuillard and Green, Fisher, Katz and Peng, and Tao.
\end{abstract}

\maketitle

\section{Introduction}

Suppose that $G$ is an abelian group and that $A$ and $A'$ are subsets of $G$. The \emph{sumset} of $A$ and $A'$ is denoted $A+A'$ and is the set of all elements of the form $a+a'$ where $a \in A$ and $a' \in A'$; more generally the $n$-fold sumset of $A$ with itself is denoted $n.A$ and defined recursively by $n.A:=A+(n-1).A$ and $1.A:=A$.

Fre{\u\i}man, in \cite{fre::0}, made a study of those sets having so called `small' sumset or doubling -- heuristically this means that $|A+A| = O(|A|)$ -- and, following popularisation by Gowers \cite{gow::4}, his work has been applied to remarkable effect in a wide range of problems.  See, for instance, \cite{szevu::,sudszevu::} and \cite{taovu::0} for some examples.

A natural class of sets with small sumset are arithmetic progressions and their generalisations.  A \emph{$d$-dimensional arithmetic progression} in $G$ is a set of the form
\begin{equation*}
\{x_0+\sum_{i=1}^d{l_i.x_i} \textrm{ where } |l_i| \leq L_i \textrm{ for all } i \in \{1,\dots,d\}\},
\end{equation*}
where $x_0,x_1,\dots,x_d \in G$ and $L_1,\dots,L_d \in \N$.  Multi-dimensional arithmetic progressions are easily seen to have small sumset, but they are not the only examples of sets with this property: if $W$ is a coset of a finite subgroup of $G$ then $|W+W| = |W|$, so it certainly has small sumset. 

Combining the aforementioned examples we define a \emph{$d$-dimensional coset progression} to be a set of the form $P+H$, where $P$ is a $d$-dimensional arithmetic progression, and $H$ is a finite subgroup of $G$.  It is easy to see that if $A$ is a `large' subset of a $d$-dimensional coset progression $M$, say $|A| \geq \delta |M|$, then
\begin{equation*}
|A+A| \leq |M+M| \leq 2^d|M| \leq 2^d\delta^{-1}|A|=O_{\delta,d}(|A|).
\end{equation*}
It turns out that these examples are exhaustive as the following theorem asserts.
\begin{theorem}[{Fre{\u\i}man's theorem for abelian groups}]\label{thm.fr}
Suppose that $G$ is an abelian group and $A \subset G$ is finite such that $|A+A| \leq K|A|$. Then there is an $O_K(1)$-dimensional coset progression $M$ containing $A$ such that $|M|=O_K(|A|)$.
\end{theorem}
This was proved for torsion-free abelian groups in \cite{fre::0}, for abelian groups of bounded exponent in \cite{ruz::01} and in general abelian groups in \cite{greruz::0}.  Considerable additional work giving alternative proofs and improving bounds has been done (see, in particular, \cite{ruz::9,bil::,cha::0} and \cite{taovu::}), however we are initially only interested in the qualitative statement.

In typical applications of Fre{\u\i}man's theorem what makes finite dimensional coset progressions important is their `group-like' properties.  These can be captured in the following more general notion. Suppose that $\rho$ is a translation invariant pseudo-metric on $G$ and write $B(\rho,\delta)$ for the ball of radius $\delta$ in $\rho$ centred at $0_G$. These balls are easily seen to be symmetric neighbourhoods of the identity.  However, they are not typically subgroups because they are not typically additively closed.

To recover a sort of `approximate closure' we introduce the following definition.  We say that a finite ball $B(\rho,\delta)$ of positive radius is \emph{$d$-dimensional} if
\begin{equation*}
|B(\rho,2\delta')| \leq 2^d|B(\rho,\delta')| \textrm{ for all } \delta' \in (0,\delta].
\end{equation*}
It follows from a Vitali-type covering argument of Bourgain \cite{bou::5} that finite dimensional balls satisfy a sort of `asymmetric approximate closure'.  Specifically,
\begin{equation*}
B(\rho,\delta) + B(\rho,\delta') \approx B(\rho,\delta+\delta')
\end{equation*}
whenever $\delta'$ is small compared to $\delta/d$.  We include the details of this argument in \S\ref{sec.pseudometricballs}; for now it suffices to know that for all practical purposes finite dimensional balls in translation invariant pseudo-metrics behave like `approximate groups'. 

It is in fact easy to see that if $B(\rho,\delta)$ is a finite dimensional ball in a translation invariant pseudo-metric then not only does it have small doubling, but it has small $n$-fold sum:
\begin{equation*}
|n.B(\rho,\delta')| \leq n^{O(d)}|B(\rho,\delta')|
\end{equation*}
whenever $\delta' \leq 2\delta/n$.  It turns out that with this slightly stronger growth condition one is able to prove the following quantitatively sharp (up to logarithmic factors) Fre{\u\i}man-type theorem.
\begin{theorem}\label{thm.weakfrei}
Suppose that $G$ is an abelian group and $A$ is a finite symmetric neighbourhood of the identity with $|n.A| \leq n^d|A|$ for all $n \geq 1$. Then $A$ is contained in an $O(d\log^32d)$-dimensional ball $B$, of positive radius, in a translation invariant pseudo-metric and $|B| \leq \exp(O(d\log 2d))|A|$.
\end{theorem}
This result is established in \cite{san::8} where it is also noted that the relative polynomial growth hypothesis is qualitatively implied by a small doubling hypothesis (\emph{c.f.} Proposition \ref{prop.done} of the appendix).  In light of this one immediately recovers a Fre{\u\i}man-type theorem suitable for applications.

Recent work in additive combinatorics has focused on extending abelian results to the non-abelian setting and it is this task to which this paper is devoted.  So far, much of the work on extending Fre{\u\i}man's theorem to non-abelian groups has concerned fairly specific examples, although the results are nevertheless impressive.  The reader may with to reflect on any of \cite{hamllaser::,elekir::,lin::,tao::6,cha::2,cha::3,hel::} and \cite{hel::0} for more details.

To be clear we recast a number of the definitions above in the non-abelian setting. Suppose that $G$ is a group and suppose that $A$ and $A'$ are subsets of $G$. The \emph{product set} of $A$ and $A'$ is denoted $A.A'$ and is the set of all elements of the form $a.a'$ where $a \in A$ and $a' \in A'$; similarly the $n$-fold product set of $A$ with itself is denoted $A^n$ and is defined recursively by $A^n:=A.A^{(n-1)}$ and $A^1:=A$.

Suppose that $\rho$ is a (left and right) translation invariant pseudo-metric on $G$.  As before $B(\rho,\delta)$ denotes the ball of radius $\delta$ in $\rho$ centred at $1_G$ and a finite ball $B(\rho,\delta)$ of positive radius is \emph{$d$-dimensional} if
\begin{equation*}
|B(\rho,2\delta')| \leq 2^d|B(\rho,\delta')| \textrm{ for all } \delta' \in (0,\delta].
\end{equation*} 
We say that a set $A \subset G$ is \emph{normal} if $xA=Ax$ for all $x \in G$, so if $G$ is abelian all sets are normal.  Any ball in a (left and right) translation invariant pseudo-metric can also be easily seen to be normal\footnote{For details see Lemma \ref{lem.trivprop}.}, so naturally enough our results will only apply to normal sets.

Finally, recall that a \emph{monomial group} is a finite group in which every irreducible representation is induced by a one-dimensional representation.  We call a group \emph{hereditarily monomial} if every subgroup is monomial -- this class includes finite nilpotent groups, but also many others and the reader is referred to \S\ref{sec.mon} for more details.  We can now state our main result.
\begin{theorem}\label{thm.maintheorem}
Suppose that $G$ is a hereditarily monomial group and $A$ is a symmetric normal neighbourhood of the identity with $|A^n| \leq n^d |A|$ for all $n \geq 1$. Then $A$ is contained in an $O(d\log^32d)$-dimensional ball $B$, of positive radius, in a (left and right) translation invariant pseudo-metric and $|B| \leq \exp(O(d\log 2d))|A|$.
\end{theorem}
It is worth making some observations.  In the first instance the hypothesis on $A$ has strong parallels with Gromov's theorem \cite{gro::} which says that any group generated by a group of polynomial growth is virtually nilpotent.  Although not actually requiring that $|A^n| \leq n^d|A|$ for all $n$, even the finitization of  Gromov's theorem requires it for some $n$ dependent on $|A|$ -- see \cite{tao::7} for details.  By contrast it is easy to read out of the proof that our theorem just requires the growth estimate for $n$ in the range
\begin{equation*}
\Omega(d\log d) = n = O(d^3 \log^5d).
\end{equation*} 
Indeed, from a certain perspective our work can be seen as attempting to develop a quantitatively effective version of Gromov's theorem.  We remark that in work in progress, Shalom and Tao \cite{shatao::} are pursuing this problem directly, and some progress has already been made by Lee and Makarychev \cite{leemak::}.

Secondly, the most immediate generalisations of $d$-dimensional arithmetic progressions to non-abelian groups do not have the small iterated-product set property and thus any potential non-abelian Fre{\u\i}man-type theorem must concern containment in some other structure.  On the other hand, finite dimensional balls in pseudo-metrics generalise directly from the abelian setting and so may be seen as a natural alternative.

The hypothesis of the theorem may be weakened, and the type of structure discovered more explicitly described at the expense of bounds and uniformity in the class of groups.  This has been explored by entirely different methods in recent work by Breuillard and Green \cite{bregre::}, Fisher, Katz and Peng \cite{fiskatpen::} and Tao \cite{tao::9}.

As mentioned above, in abelian groups it is easy to pass from a small doubling hypothesis to a relative polynomial growth hypothesis.  In non-abelian groups this is not possible, however it turns out that in nilpotent groups small tripling does imply relative polynomial growth as we show in the appendix.  Thus, just as one may recover a version of Theorem \ref{thm.fr} for pseudo-metric balls from Theorem \ref{thm.weakfrei} and a covering argument, so we recover the following theorem from Theorem \ref{thm.maintheorem}. 
\begin{theorem}\label{thm.mainnilpotenttheorem}
Suppose that $G$ is a finite nilpotent group of class $c$ and $A$ is a symmetric normal neighbourhood of the identity with $|A^3| \leq K |A|$ for all $n \geq 1$. Then $A$ is contained in an $O_{K,c}(1)$-dimensional ball $B$, of positive radius, in a (left and right) translation invariant pseudo-metric and $|B| \leq O_{K,c}(|A|)$.
\end{theorem}
A couple of remarks are in order: the insistence that the group be finite is not necessary in that a modelling argument can be used to remove it.  This is not a short procedure though because the ball found in Theorem \ref{thm.maintheorem} does not naturally live inside a small power of $A$.  To put it inside one has to use a Bogolbyubov-type argument in the style of Ruzsa \cite{ruz::9} which considerably extends the argument for a rather small benefit.

The second remark is that the final bounds \emph{do} depend on the nilpotency class of the group despite no similar sort of dependence in Theorem \ref{thm.maintheorem}.  This is because small tripling entails polynomial growth of a degree dependent on the class of the group and it is this growth rate which appears in the hypotheses of Theorem \ref{thm.maintheorem}.

The proof of Theorem \ref{thm.maintheorem} is an attempt to generalise the arguments of \cite{san::8} as much as possible and follows, in spirit, the argument there. The paper splits as follows. In the next section, \S\ref{sec.fourier}, we record the representation theory that we require; following this, in \S\ref{sec.mon}, we give a discussion of monomial and hereditarily monomial groups including examples of better known classes which fit under this umbrella. \S\ref{sec.pseudometricballs} then develops the basic properties of pseudo-metric balls useful in general and for applications, before we are ready to prove our main theorem in \S\S \ref{sec.bohrsetsandbourgainsystems}--\ref{sec.proofoftheorem}.  

\section{Representation theory and the {F}ourier transform}\label{sec.fourier}

In this section we standardise some basic definitions with regard to representation theory and the Fourier transform. There are many books covering our needs in the case of finite groups, but we found the notes \cite{tao::2} of Tao to be a good introduction.

While the extension of our work from finite to compact groups is not difficult (in light of Tannaka-Krein duality), the extension to locally compact groups already faces considerable additional problems because there is not even a clear description of Plancherel measure in this case. The interested reader may wish to consult the books \cite{nauste::} or \cite{kir::} for more details.

\subsection{Some matrix groups} We write $M_n(\C)$ for the $C^*$-algebra of matrices on the Hilbert space $\C^n$ endowed with the usual $C^*$-norm, the \emph{spectral radius}:
\begin{equation*}
|M|:=\sup\{|\lambda|: \lambda \textrm{ is an eigenvalue of }M\}.
\end{equation*}
Since $\C^n$ is finite dimensional $M_n(\C)$ is also an inner product space with the Hilbert-Schmidt inner product defined by
\begin{equation*}
\langle M,M' \rangle:=\tr M^*M'\textrm{ and, hence, } \|M\|:=\sqrt{\langle M,M\rangle}.
\end{equation*}
Here, of course, $\tr M$ denotes the trace of the matrix $M$. 

We write $U_n(\C)$ for the group of unitary matrices on $\C^n$, and $I_n$ for the identity matrix in $U_n$, frequently dropping the subscript when the dimension is clear.

\subsection{Representations}
A \emph{representation} of a finite group $G$ is a homomorphism $\gamma:G \rightarrow U_n(\C)$, where we call $n$ the \emph{dimension} of the representation; unless it is otherwise declared we write $d_\gamma$ for this quantity.

Two representations $\gamma$ and $\gamma'$ are \emph{equivalent}, denoted $\gamma \cong \gamma'$, if there is a unitary matrix $U$ such that
\begin{equation*}
U\gamma(x)=\gamma'(x)U \textrm{ for all } x \in G.
\end{equation*}

The operations of direct sum and tensor product of matrices carry over to give us a sum and product. If $\gamma$ and $\gamma'$ are two representations then their direct sum is the homomorphism
\begin{equation*}
\gamma+\gamma': G \rightarrow U_{d_\gamma+d_{\gamma'}}(\C); x \mapsto \gamma(x) \oplus \gamma'(x)
\end{equation*}
under the usual embedding of $U_{d_\gamma}(\C) \oplus U_{d_{\gamma'}}(\C)$ into $U_{d_{\gamma}+d_{\gamma'}}(\C)$. Similarly their tensor product is the homomorphism
\begin{equation*}
\gamma\gamma': G \rightarrow U_{d_\gamma d_{\gamma'}}(\C); x \mapsto \gamma(x) \otimes \gamma'(x)
\end{equation*}
under the usual embedding of $U_{d_\gamma}(\C) \otimes U_{d_{\gamma'}}(\C)$ into $U_{d_{\gamma}d_{\gamma'}}(\C)$. We make the cautionary remark that in general $(\gamma+\gamma')(x)$ is not equal to $\gamma(x) + \gamma'(x)$ and $(\gamma\gamma')(x)$ is not equal to $\gamma(x)\gamma'(x)$.

Finally it is useful to understand representations through their building blocks: we call a representation $\gamma$  \emph{irreducible} if it may not be written as $\gamma' + \gamma''$ for some non-zero representations, and we write $\wh{G}$ for the set of (equivalence classes of) irreducible representations of $G$.

\subsection{The Fourier transform} 
Given a finite group $G$, we write $\P_G$ for the unique Haar probability measure on $G$, i.e. the measure assigning mass $|G|^{-1}$ to each element of $G$, and $\E$ for the corresponding expectation operator.

We are now ready to define the Fourier transform which takes $f \in L^1(\P_G)$ to $\wh{f}$ defined at the representation $\gamma$ by
\begin{equation*}
\wh{f}(\gamma):=\E_{x \in G}{f(x)\gamma(x)}.
\end{equation*}
The Fourier transform is particularly important in view of its action on convolution. Recall that if $f,g \in L^1(\P_G)$  then their convolution is the function $f \ast g$ defined by
\begin{equation*}
f \ast g(x):=\E_{y \in G}{f(y)g(y^{-1}x)} \textrm{ for all } x \in
G;
\end{equation*}
any easy computation verifies that $\wh{f \ast g}(\gamma)=\wh{f}(\gamma).\wh{g}(\gamma)$.

Given $f \in L^1(\P_G)$ we write $\tilde{f}$ for the function $x \mapsto \overline{f(x^{-1})}$ and say that $f$ is \emph{hermitian} if $\tilde{f}=f$. The reason for this definition is that if $f$ is hermitian and $\gamma$ is a representation then $\wh{f}(\gamma)$ is hermitian as can be seen by a short calculation:
\begin{equation*}
\wh{f}(\gamma)^* = \E_{x \in G}{\overline{f(x)}\gamma(x)^*}=\E_{x \in G}{\overline{f(x)}\gamma(x^{-1})} = \E_{y \in G}{\overline{f(y^{-1})}\gamma(y)} = \wh{f}(\gamma).
\end{equation*}

Of central importance to us is the non-abelian version of Plancherel's theorem which we record now. Given functions $f,g \in L^2(\P_G)$ we have that
\begin{equation*}
\langle f,g\rangle_{L^2(\P_G)} = \sum_{\gamma \in \wh{G}}{d_\gamma \langle \wh{f}(\gamma),\wh{g}(\gamma)\rangle}.
\end{equation*}
Note that here we make a common abuse (which we shall repeat throughout the paper) of writing $\wh{f}(\gamma)$ when $\gamma$ is, in fact, an equivalence class of representations.  However, it is easy to see that if $\gamma'$ and $\gamma''$ are equivalent representations then
\begin{equation*}
\langle \wh{f}(\gamma'),\wh{g}(\gamma')\rangle=\langle \wh{f}(\gamma''),\wh{g}(\gamma'')\rangle,
\end{equation*}
so there is no real ambiguity in this.  In general, whenever a choice of which element of an equivalence class of representations needs to be made the eventual result will not depend on that choice and we shall make no further mention.

\subsection{Character theory and induction}

Suppose that $G$ is a finite group and $\gamma$ is a representation of $G$. We write $\chi_\gamma$ for the \emph{character} on $G$ corresponding to $\gamma$, that is the map $x \mapsto \tr \gamma(x)$; it is easy to see from this definition that $\chi_\gamma(1_G)=d_\gamma$.  Characters encode all the representation theory of a group since $\chi_{\gamma} = \chi_{\gamma'}$ if and only if $\gamma \cong \gamma'$.

A function $f$ on $G$ is called a \emph{class function} if $f$ is constant on conjugacy classes, so $f(xyx^{-1})=f(y)$ for all $x,y \in G$.  We write $\Class(G)$ for the set of all complex valued class functions, which naturally forms a Hilbert space for which the set of characters are an orthonormal basis.  Our interest, however, is in their r{\^o}le in induction. We shall record the basic facts now; the reader looking for details may consult \cite[Chapter 5]{isa::}.

Suppose that $H \leq G$ and $f \in \Class(H)$. Then we define the \emph{induced class function} on $G$, denoted $f^G$, by
\begin{equation*}
f^G(x)=\P_G(H)^{-1}\E_{y \in G}{f(yxy^{-1})},
\end{equation*}
where $f(z)$ is understood to be zero unless $z \in H$. It is immediate from the definition that $f^G \in \Class(G)$ and we have the following useful result called Frobenius' reciprocity theorem. Suppose that $f \in \Class(H)$ and $g \in \Class(G)$ then
\begin{equation*}
\langle f,g|_H\rangle_{L^2(\P_H)} = \langle f^G, g \rangle_{L^2(\P_G)}. 
\end{equation*}
A consequence of reciprocity is that if $\chi$ is a character on $H$, then $\chi^G$ is a character on $G$, called the \emph{induced character}.  Naturally, we say that a representation $\gamma$ is \emph{induced} by a representation $\gamma'$ if $\chi_\gamma$ is induced by $\chi_{\gamma'}$, and it is then easy to see that
\begin{equation*}
d_\gamma = \chi_{\gamma'}^G(1_G) = \P_G(H)^{-1}\chi_{\gamma'}(1_H) = \P_G(H)^{-1}d_{\gamma'}.
\end{equation*}
This fact will be useful later.

\subsection{One-dimensional representations}

One-dimensional representations have particularly special properties.  Suppose that $\gamma$ is a one-dimensional representation.  Then, since $U_1(\C) = S^1$, $\gamma$ is a homomorphism from $G$ into $S^1$.

We write $\Lin(G)$ for the set of one-dimensional representations of $G$ or, equivalently, the set of homomorphisms from $G$ to $S^1$. It is easy to see that $\Lin(G)$ forms an abelian group and to emphasise this we shall denote the group operation (which is really the tensor product) by `$+$'.  It should be noted that if $G$ is abelian then $\Lin(G)=\wh{G}$.

\section{Monomial and hereditarily monomial groups: examples}\label{sec.mon}

Recall from the introduction that a monomial group is a group in which every representation is induced by a one-dimensional representation, and a group is hereditarily monomial if every subgroup is monomial.  

Monomial groups have received considerable attention from the group theoretic community and for a fairly detailed discussion the reader may wish to consult \cite[\S24]{hup::} or the book \cite{isa::}. To get a sense of what is known about monomial groups we relate the definition to some more common classes of groups.

Recall that a group $G$ is said to be \emph{solvable} if there is a normal series
\begin{equation*}
G \rhd G_0 \rhd G_1 \rhd \dots \rhd G_d = \{1_G\}
\end{equation*}
such that $G_i/G_{i+1}$ is abelian for all $i$.  Taketa gave a straightforward proof that all monomial groups are solvable (see \cite[Theorem 24.4]{hup::}). Unfortunately, there are examples of finite solvable groups which are not monomial (see \cite[Example 24.16]{hup::}).  Curiously, however, Dade \cite[Theorem 25.9]{hup::} showed that every solvable group is a subgroup of a monomial group which shatters the hope of a simple structural characterization of monomial groups.   Of course, it follows that not every monomial group is hereditarily monomial.

On the other hand there are a number of classes of groups which \emph{are} hereditarily monomial.  Recall that a \emph{supersolvable} group is one for which there is a normal series
\begin{equation*}
G \rhd G_0 \rhd G_1 \rhd \dots \rhd G_d = \{1_G\}
\end{equation*}
such that $G \rhd G_i$ and $G_i/G_{i+1}$ is cyclic for all $i$. This class includes nilpotent groups (see \cite{bradesjoh::}). It turns out that every finite supersolvable group is monomial \cite[Corollary 2.3.5]{bradesjoh::}, and the class is closed under taking subgroups so in fact every finite supersolvable group is hereditarily monomial.

In a different direction solvable $A$-groups have also received attention in the literature. An \emph{$A$-group} is a group all of whose Sylow subgroups are abelian. As with supersolvable groups, every finite solvable $A$-group is monomial \cite[Corollary 2.3.10]{bradesjoh::}.  Again, this class is closed under taking subgroups, whence every finite solvable $A$-group is hereditarily monomial.

In general monomial groups have resisted a simple non-representation theoretic description and so our results apply to a wider class of groups than can be easily characterised.

\section{Finite dimensional balls in translation invariant pseudo-metrics}\label{sec.pseudometricballs}

Suppose that $G$ is a group and $\rho$ is a (left and right) invariant pseudo-metric on $G$.  As mentioned in the introduction balls in $\rho$ are a natural candidate for approximate groups, and in this section we shall develop some of their basic properties as well as giving some examples.

Before we begin, as further evidence of their utility we remark that when $G$ is abelian there are many arguments which employ finite dimensional balls in pseudo-metrics which exploit the specific pseudo-metric in question, but where in fact the only properties one needs hold for any pseudo-metric.  The interested reader is referred to any of \cite{gow::4,bou::5,boukattao::,gre::02,shk::0,shk::1} and \cite{gretao::1} for examples of this.

\begin{lemma}[Properties of balls in translation invariant pseudo-metrics]\label{lem.trivprop} Suppose that $G$ is a group and $\rho$ is a (left and right) translation invariant pseudo-metric on $G$. Then
\begin{enumerate}
\item (Symmetric neighborhood) $B(\rho,\delta)$ is a symmetric neighborhood of the identity for all $\delta \in \R^+$;
\item (Nesting) $B(\rho,\delta') \subset B(\rho,\delta)$ for all $\delta,\delta' \in \R^+$ with $\delta' \leq \delta$;
\item (Subadditivity) $B(\rho,\delta).B(\rho,\delta') \subset B(\rho,\delta+\delta')$ for all $\delta,\delta' \in \R^+$;
\item (Normality) $xB(\rho,\delta)=B(\rho,\delta)x$ for all $x \in G$ and $\delta \in \R^+$. 
\end{enumerate}
\end{lemma}
\begin{proof}
That the balls $B(\rho,\delta)$ are symmetric neighborhoods of the identity follows from the symmetry property of $\rho$ and the fact that $\rho(1_G,1_G)=0$; also, nesting is immediate.

Subadditivity follows from the triangle inequality for $\rho$: suppose that $x \in B(\rho,\delta)$ and $x' \in B(\rho,\delta')$.  Then 
\begin{equation*}
\rho(xx',1_G) \leq \rho(xx',x') + \rho(x',1_G) = \rho(x,1_G) + \rho(x',1_G) \leq \delta+\delta',
\end{equation*} 
where the intermediate equality is by right invariance of $\rho$. It follows that
\begin{equation*}
B(\rho,\delta).B(\rho,\delta') \subset B(\rho,\delta+\delta')
\end{equation*}
as required.

The normality condition follows from the fact that $\rho$ is left \emph{and} right invariant.  By left and then right invariance we have
\begin{equation*}
\rho(x^{-1}yx,1_G) = \rho(yx,x)=\rho(y,1_G).
\end{equation*}
Whence $y \in B(\rho,\delta)$ if and only if $x^{-1}yx \in B(\rho,\delta)$, i.e. $xB(\rho,\delta)x^{-1}=B(\rho,\delta)$ as required.
\end{proof}
Now, when a ball $B(\rho,\delta)$ is finite dimensional it is possible to recover a sort of approximate additive closure on average as observed by Bourgain in \cite{bou::5} -- the proof goes through in the non-abelian setting and is the content of the following proposition.  The result is not logically necessary for the paper but is likely to be for applications.
\begin{proposition}
Suppose that $G$ is a group, $\rho$ is a (left and right) translation invariant pseudo-metric on $G$ and $B(\rho,\delta)$ is a finite $d$-dimensional ball.  Then there is a $\lambda \in (1,2]$ such that
\begin{equation*}
1-6d|\eta| \leq \frac{|B(\rho,\lambda\delta(1 + \eta)))|}{|B(\rho,\lambda\delta)|}  \leq 1+6d|\eta|
\end{equation*}
whenever $|\eta| \leq 1/6d$.
\end{proposition}
\begin{proof}
Let $f : [0,1] \rightarrow \R$ be the function 
\begin{equation*}
f(\alpha) :=
\frac{1}{d}\log_2 |B(\rho,2^{1-\alpha}\delta)|
\end{equation*}
and note that $f$ is non-decreasing in $\alpha$ with $f(1) - f(0) \leq 1$ since $B(\rho,\delta)$ is $d$-dimensional. We claim that there is an $\alpha \in [\frac{1}{6}, \frac{5}{6}]$ such that
\begin{equation}\label{eqn.s}
|f(\alpha+ x) - f(\alpha)| \leq 3|x| \textrm{ for all } |x| \leq 1/6.
\end{equation}
If no such $\alpha$ exists then for every $\alpha \in
[\frac{1}{6}, \frac{5}{6}]$ there is an interval $I(\alpha)$ of
length at most $\frac{1}{6}$ having one endpoint equal to $\alpha$
and with
\begin{equation*}
\int_{I(\alpha)}{ df} > \int_{I(\alpha)}{ 3 dx}.
\end{equation*}
These intervals cover $[\frac{1}{6}, \frac{5}{6}]$, which has total length
$\frac{2}{3}$. A simple covering lemma allows us to pass to a
disjoint subcollection $I_1 \cup ... \cup I_n$ of these intervals
with total length at least $\frac{1}{3}$. However we now have
\begin{equation*}
1 \geq \int^1_0 df \geq \sum_{i=1}^n \int_{I_i} df > \sum_{i = 1}^n
\int_{I_i} 3 \, dx \geq 1,
\end{equation*}
a contradiction. It follows that there is an $\alpha$ such that (\ref{eqn.s}) holds. Setting $\lambda := 2^{1-\alpha}$, it is easy to see that
\begin{equation*}
(1+|\eta|)^{-3d} \leq \frac{|B(\rho,\lambda\delta(1+\eta))|}{|B(\rho,\lambda\delta)|} \leq (1+|\eta|)^{3d}
\end{equation*}
whenever $|\eta| \leq 1/6$. But if $3d|\eta| \leq 1/2$ then
\begin{equation*}
(1+|\eta|)^{3d} \leq 1+6d|\eta| \textrm{ and } (1+|\eta|)^{-3d} \geq
1-6d|\eta|.
\end{equation*}
The result follows.
\end{proof}
Having established some basic properties of translation invariant pseudo-metrics it is instructive to consider some examples.
\begin{example}[Linear Bohr sets]\label{ex.bohr}
Suppose that $G$ is a group and $\Gamma$ is a finite subset of $\Lin(G)$. There is a natural (left and right) translation invariant pseudo-metric associated to $\Gamma$ as we shall now explain.

For any $z \in S^1$ write $\|z\|$ for the quantity $(2\pi)^{-1}|\Arg z|$, where the argument is taken to have a value lying in $(-\pi,\pi]$.  We define our pseudo-metric by
\begin{equation*}
\rho(x,y):=\sup{\{\| \gamma(xy^{-1})\|: \gamma \in \Gamma\}},
\end{equation*}
 and it is easy to check that this is a (left and right) translation invariant pseudo-metric on $G$. A \emph{linear Bohr set} is a ball in this pseudo-metric \emph{viz}.
\begin{equation*}
\Bohr(\Gamma,\delta):=B(\rho,\delta).
\end{equation*}
Note that this collapses to the usual definition of a Bohr set when the group is abelian.

In \cite[Lemma 4.19]{taovu::} it is shown that Bohr sets in abelian groups are $O(|\Gamma|)$-dimensional balls and the same argument can be used here.

For each $\theta \in \T^\Gamma$ define the set
\begin{equation*}
B_\theta:=\left\{x \in G: \|\gamma(x) \exp(-2\pi i \theta_\gamma)\|
\leq \delta/2\textrm{ for all } \gamma \in \Gamma\right\}.
\end{equation*}
If $B_\theta$ is non-empty let $x_\theta$ be some member. The map $x
\mapsto x_\theta^{-1}x$ is an injection from $B_\theta$ into
$\Bohr(\Gamma,\delta)$, so putting $T_\delta:=\{x_\theta: \theta \in
\prod_{\gamma \in
\Gamma}{\{-3\delta/2,-\delta/2,\delta/2,3\delta/2\}}\}$ we have that
\begin{equation*}
\Bohr(\Gamma,2\delta) \subset T_\delta .\Bohr(\Gamma,\delta).
\end{equation*}
Since $|T_\delta| \leq 4^{|\Gamma|}$ we get that $\Bohr(\Gamma,\delta)$ is an
$O(|\Gamma|)$-dimensional ball.
\end{example}
We should remark that there is a more general notion of Bohr set in non-abelian groups which does not require that the characters all be linear; hence the nomenclature.
\begin{example}[Large spectra] Suppose that $G$ is a finite group and $A$ is a subset of $G$.  We define a translation invariant pseudo-metric on $\Lin(G)$ by
\begin{equation*}
\rho(\gamma,\gamma'):=\|\gamma-\gamma'\|_{L^2(\P_G(A)^{-1}1_A \ast 1_{-A})},
\end{equation*}
and write $\LSpec(A,\delta)$ for the ball of radius $\delta$ in $\rho$,
calling such sets \emph{large spectra}. The true utility of this definition emerges when one notes that
\begin{equation*}
\rho(0_{\Lin(G)},\gamma) = 2(1-\P_G(A)^{-2}|\wh{1_A}(\gamma)|^2).
\end{equation*}
To see this, recall that $0_{\Lin(G)}$ is the representation which has $0_{\Lin(G)}(x)\equiv 1$, so
\begin{eqnarray*}
\|1-\gamma\|_{L^2(\P_G(A)^{-1}1_A \ast 1_{-A})}^2 &= & \P_G(A)^{-2}\E_{x \in A}{(2-\gamma(x)-\overline{\gamma(x)})1_A \ast 1_{-A}(x)}\\ & = &2 - \P_G(A)^{-2}(|\wh{1_A}(\gamma)|^2 +|\wh{1_A}(\overline{\gamma})|^2),
\end{eqnarray*}
which equals the desired expression. Rearranging this tells us that
\begin{equation*}
\LSpec(A,\delta)=\{\gamma \in \Lin(G):|\wh{1_A}(\gamma)| \geq \sqrt{1-\delta^2/2}\P_G(A)\}.
\end{equation*}
It is, of course, this fact which motivates the name `large spectrum'.
\end{example}
There is no useful general growth estimate for large spectra in the way that there is for linear Bohr sets, although the content of \S\ref{sec.growthoflargespectra} is to show that if $A$ satisfies a growth hypothesis then there is.

\section{Growth of linear Bohr sets}\label{sec.bohrsetsandbourgainsystems}

In this section we estimate the growth of linear Bohr sets when their defining set of representations is structured.  The proposition should be compared with the `trivial estimate' for the dimension given in Example \ref{ex.bohr}.
\begin{proposition}\label{prop.approximateannihilatorsofapproximategroups}
Suppose that $G$ is a finite group, $\Gamma$ is a symmetric subset of $\Lin(G)$ containing the identity with $\Gamma+\Gamma \subset \Span(X)+\Gamma$ for some finite set
$X$, and $\delta \in (0,2^{-4}]$ is a parameter. Then
\begin{equation*}
\P_G(\Bohr(\Gamma \cup X,2\delta)) \leq \exp(O(|X|\log
2|X|))\P_G(\Bohr(\Gamma \cup X,\delta)).
\end{equation*}
\end{proposition}
We require a preliminary result. Suppose that $\Lambda \subset \Lin(G)$, $k$ is a positive integer and $\delta \in (0,1]$. By the
triangle inequality it is immediate that $\Bohr(\Lambda,\delta)
\subset \Bohr(k\Lambda,k\delta)$; the following elementary lemma can be used to provide a
partial converse.
\begin{lemma}
Suppose that $t$ is a real number, $k$ is a positive integer, $\delta \in (0,1]$ has $k\delta<1/3$ and\footnote{Here $\langle x \rangle$ denotes the distance from $x$ to the nearest integer.} $\langle rt\rangle \leq k\delta$ for all $r \in \{1,\dots,k\}$. Then $\langle t \rangle \leq \delta$.\hfill $\Box$
\end{lemma}
\begin{corollary}\label{lem.technicallemma}
Suppose that $\Lambda\subset \Lin(G)$ contains the identity and $k\delta < 1/3$. Then $\Bohr(k\Lambda,k\delta) \subset \Bohr(\Lambda,\delta)$, and hence $\Bohr(k\Lambda,k\delta) = \Bohr(\Lambda,\delta)$. 
\end{corollary}
\begin{proof}
Since $0_{\Lin(G)} \in \Lambda$, we have that $r\lambda \in k\Lambda$
for all $r \in \{1,\dots,k\}$. It follows that if $x \in
\Bohr(k\Lambda,k\delta)$ then
\begin{equation*}
\|\lambda(x)^r\|=\|(r\lambda)(x)\| \leq k \delta \textrm{ for all }
r \in \{1,\dots,k\}.
\end{equation*}
If we define $\theta_x \in (-1,1]$ to be such that
$\lambda(x)=\exp(i \pi \theta_x)$, then we can rewrite the above as
\begin{equation*}
\langle r\theta_x\rangle \leq k\delta \textrm{ for all } r \in
\{1,\dots,k\}.
\end{equation*}
It follows from the preceding lemma that $\langle \theta_x \rangle \leq \delta$ and hence that $x \in \Bohr(\Lambda,\delta)$. The result is proved.
\end{proof}
\begin{proof}[Proof of Proposition
\ref{prop.approximateannihilatorsofapproximategroups}] For each
$\theta \in \T^X$ define the set
\begin{equation*}
B_\theta:=\left\{x \in G: \|\gamma(x) \exp(-2\pi i \theta_\gamma)\|
\leq \delta/2^2|X|\textrm{ for all } \gamma \in X\right\}.
\end{equation*}
Put $I:=\{k\delta/2^2|X|: -2^4|X| \leq k \leq 2^4|X|\}$ and note
that
\begin{equation*}
\Bohr(\Gamma \cup X,2\delta) \subset \bigcup\{B_\theta \cap
\Bohr(\Gamma,2\delta) : \theta \in I^X\}.
\end{equation*}
For each $\theta \in I^X$ let $x_\theta$ be some element of
$B_\theta \cap \Bohr(\Gamma,2\delta)$ (if the set is non-empty); the
map $x \mapsto x_\theta^{-1}x$ is an injection from $B_\theta \cap
\Bohr(\Gamma,2\delta)$ into $\Bohr(X,\delta/2|X|)\cap \Bohr(\Gamma,2^2\delta)$. Writing $T$ for the set of all such
$x_\theta$s, we have
\begin{equation*}
\Bohr(\Gamma \cup X,2\delta)  \subset T.(\Bohr(\Gamma,2^2\delta)\cap
\Bohr(X,\delta/2|X|))
\end{equation*}
Now, by the triangle inequality we have
\begin{equation*}
\Bohr(\Gamma,2^2\delta) \cap \Bohr(X,\delta/2|X|) \subset
\Bohr(\Gamma +2^3\Span(X),2^3 \delta),
\end{equation*}
and since the identity is in $\Gamma$ and $\Gamma+\Gamma \subset
\Gamma + \Span(X)$ we have $\Gamma+2^3\Span(X) \supset 2^3\Gamma$
and $\Gamma+2^3\Span(X) \supset 2^3\Span(X)$, whence
\begin{equation*}
\Bohr(\Gamma +2^3\Span(X),2^3 \delta) \subset
\Bohr(2^3\Gamma,2^3\delta)\cap\Bohr(2^3\Span(X),2^3\delta).
\end{equation*}
Finally, by Corollary \ref{lem.technicallemma} and the fact that $X
\subset \Span(X)$ we have
\begin{equation*}
\Bohr(2^3\Gamma,2^3\delta)\cap\Bohr(2^3\Span(X),2^3\delta) \subset
\Bohr(\Gamma,\delta) \cap \Bohr(X,\delta)
\end{equation*}
and the result follows on noting that $|T| \leq |I|^{|X|}$.
\end{proof}

\section{Representations supporting very large values of the Fourier transform}

In this section we consider those representations $\gamma$ for which $|\wh{1_A}(\gamma)|$ is very large.  It will turn out that under certain conditions they are, in fact, one-dimensional.

We begin with a lemma which says that the Fourier transform is particularly simple when the function in question is a class function.  This observation was introduced by Lubotzky, Phillips and Sarnak in \cite{lubphisar::} (see also \cite{davsarval::}) and then leveraged by Gowers in an additive combinatorial setting in \cite{gow::2}.
\begin{lemma}\label{lem.sarnak}
Suppose that $G$ is a finite group, $\gamma$ is an irreducible representation of $G$ and $f$ is a hermitian, complex-valued class function on $G$. Then $\wh{f}(\gamma)$ is a scalar multiple of the identity.
\end{lemma}
\begin{proof}
Since $f$ is hermitian, $\wh{f}(\gamma)$ is hermitian and hence $\C^{d_\gamma}$ has an eigenvector $v \in \C^{d_\gamma}$ with eigenvalue $\mu$.  As a consequence of the fact that $f$ is a class function we shall show that $\gamma(y)v$ is also an eigenvector of $\wh{f}(\gamma)$ with eigenvalue $\mu$.

First, note that
\begin{equation*}
\wh{f}(\gamma)\gamma(y)v = \E_{x \in G}{f(x)\gamma(x)\gamma(y)v} = \E_{x \in G}{f(x)\gamma(xy)v},
\end{equation*}
since $\gamma$ is a homomorphism.  Now, by the change of variables $u=y^{-1}xy$ we get that
\begin{equation*}
\wh{f}(\gamma)\gamma(y)v=\E_{u \in G}{f(yuy^{-1})\gamma(yu)v} = \gamma(y)\E_{u \in G}{f(yuy^{-1})\gamma(u)v},
\end{equation*}
again since $\gamma$ is a homomorphism.  However, $f$ is a class function so $f(yuy^{-1})=f(u)$, whence
\begin{equation*}
\wh{f}(\gamma)\gamma(y)v = \gamma(y)\E_{u \in G}{f(u)\gamma(u)v} = \gamma(y)\wh{f}(\gamma)v = \gamma(y)(\mu v)=\mu \gamma(y)v.
\end{equation*}
Thus $\gamma(y)v$ is an eigenvector of $\wh{f}(\gamma)$ with eigenvalue $\mu$ as claimed.

The subspace of $\C^{d_\gamma}$ generated by $\{\gamma(y)v: y \in G\}$ is trivially invariant and hence, since $\gamma$ is irreducible, is the whole of $\C^{d_\gamma}$.  However, every element of this subspace is an eigenvector of $\wh{f}(\gamma)$ with eigenvalue $\mu$, whence $\wh{f}(\gamma)=\mu I$.  The lemma is proved.
\end{proof}
The preceding lemma has two applications.  The first is as follows and tells us that the only representations of importance in our study are the one-dimensional ones.
\begin{lemma}\label{lem.nxt}
Suppose that $G$ is a finite group generated by a set $S$ containing the identity, $\gamma$ is an irreducible monomial representation of $G$, $A$ is a normal symmetric subset of $G$ such that\begin{equation*}
2|\wh{1_A}(\gamma)| > \P_G(S.A).
\end{equation*}
Then $\gamma$ is one-dimensional.
\end{lemma}
\begin{proof}
By linearity of trace we have that
\begin{equation*}
\langle \chi_\gamma,1_A \rangle_{L^2(\P_G)}= \E_{x \in G}{\tr \gamma(x)1_A(x)} = \tr \E_{x \in G}{1_A(x)\gamma(x)} = \tr \wh{1_A}(\gamma).
\end{equation*}
$A$ is symmetric and normal, so $1_A$ is a complex-valued, hermitian class function.  Furthermore, $\gamma$ is irreducible, so by Lemma \ref{lem.sarnak}, we conclude that $\wh{1_A}(\gamma)$ is a scalar multiple of the identity.  This, in the previous, tells us that
\begin{equation*}
\langle \chi_\gamma,1_A \rangle_{L^2(\P_G)} = \tr \wh{1_A}(\gamma) = d_\gamma \mu
\end{equation*}
for some scalar $\mu$ with $|\mu| = |\wh{1_A}(\gamma)|$.

On the other hand, $\gamma$ is monomial, whence there is a subgroup $H \leq G$ with index $d_\gamma$ and a homomorphism $\lambda:H \rightarrow S^1$ such that
\begin{equation*}
\langle \chi_\gamma,f \rangle_{L^2(\P_G)} = \langle \lambda,f|_H\rangle_{L^2(\P_H)}
\end{equation*}
whenever $f\in \Class(G)$ by Frobenius' reciprocity theorem.  Since $1_A \in \Class(G)$, we conclude that
\begin{equation*}
d_\gamma \mu = \langle \lambda, 1_{A \cap H} \rangle_{L^2(\P_H)}.
\end{equation*}
But $|\lambda| \leq 1$ point-wise, so
\begin{equation*}
d_\gamma |\wh{1_A}(\gamma)| \leq \P_H(A \cap H) = \P_G(A \cap H). \P_G(H)^{-1}.
\end{equation*}
However, $d_\gamma= \P_G(H)^{-1}$ since $\lambda$ is one-dimensional. Cancelling this term from both sides (possible since subgroups are non-empty) tells us that
\begin{equation*}
|\wh{1_A}(\gamma)| \leq \P_G(A \cap H).
\end{equation*}

Now, suppose that $\gamma$ is not one-dimensional so that $H$ is a proper subgroup of $G$.  In that case $S. H \neq H$, and so there is some element $s \in S$ such that $s.H \cap H = \emptyset$.  Since $1_G$ is also in $S$ it follows that 
\begin{equation*}
S.(A \cap H) \supset (A \cap H) \cup s.(A \cap H),
\end{equation*}
and this union is disjoint. We conclude that
\begin{equation*}
\P_G(S.A) \geq 2\P_G(A \cap H) \geq 2|\wh{1_A}(\gamma)|;
\end{equation*}
a contradiction to the hypothesis of the lemma.  Thus $H$ could not have been a proper subgroup of $G$ and hence $\gamma$ is one-dimensional as required.
\end{proof}

\section{Size and growth of large spectra}\label{sec.growthoflargespectra}

In this section we establish that there are many representations in the large spectrum and that it satisfies some growth estimates.  In essence the idea of leveraging relative polynomial growth to get very large values of the Fourier transform was introduced by Schoen in \cite{sch::0}, and it is this technique to which we appeal.

There is some notation and certain hypotheses which are common to all the results of this section. We record these now and they will be assumed throughout.  
\begin{enumerate}
\item $G$ is a monomial group generated by a set $S$ containing the identity;
\item $A$ is a symmetric normal subset of $G$;
\item $\P_G(S.A) < \sqrt{2}\P_G(A)$;
\item $f$ denotes the $k$-fold convolution of $1_A$ with itself.
\end{enumerate}
These hypotheses are there so that we may apply Lemma \ref{lem.nxt}.

We begin with a preliminary lemma useful for establishing both the size and growth estimates.
\begin{lemma}\label{lem.prelim}
Suppose that $\eta \in (0,1]$ and $k \in \Z$ are such that
\begin{equation*}
(1-\eta^2/2)^{k-1} \leq \P_G(A)/2\P_G(A^k).
\end{equation*}
Then
\begin{equation*}
\sum_{\gamma \in \LSpec(A,\eta)}{d_\gamma^2|\wh{1_A}(\gamma)|^{2k}} \geq \frac{1}{2}\sum_{\gamma \in \wh{G}}{d_\gamma^2|\wh{1_A}(\gamma)|^{2k}} \geq \frac{\P_G(A)^{2k}}{2\P_G(A^{k})}
\end{equation*}
\end{lemma}
\begin{proof}
Suppose that $\gamma$ is irreducible.  Since $A$ is symmetric and normal, $1_A$ is a hermitian, complex-valued class function and we may apply Lemma \ref{lem.sarnak} to get that $\wh{1_A}(\gamma)=\mu I$, so $|\mu| = |\wh{1_A}(\gamma)|$.  Two consequences follow: first,
\begin{equation}\label{eqn.l2}
\|\wh{1_A}(\gamma)\|^2=|\mu|^2.d_\gamma = d_\gamma |\wh{1_A}(\gamma)|^2;
\end{equation}
secondly, since $\wh{f}(\gamma) = \wh{1_A}(\gamma)^k=\mu^k I$,
\begin{equation*}
\|\wh{f}(\gamma)\|^2 = |\mu^k|^2.d_\gamma = d_\gamma|\wh{1_A}(\gamma)|^{2k}.
\end{equation*}

Now, by Parseval's theorem and the Cauchy-Schwarz inequality we have
\begin{equation}\label{eqn.polynomialgrowthCS}
\sum_{\gamma \in \wh{G}}{d_\gamma \|\wh{f}(\gamma)\|^2}= \E_{x \in G}{f^2(x)}
\geq\frac{1}{\P_G(\supp f)}(\E_{x \in G} f(x))^2 =
\frac{\P_G(A)^{2k}}{\P_G(A^k)},
\end{equation}
whence
\begin{equation}\label{eqn.polCS2}
\sum_{\gamma \in \wh{G}}{d_\gamma^2|\wh{1_A}(\gamma)|^{2k}} \geq \frac{\P_G(A)^{2k}}{\P_G(A^k)}.
\end{equation}

We split the range of integration on the left into the set of representations where $|\wh{1_A}(\gamma)|$ is large and where it is small.  Specifically, let 
\begin{equation*}
\mathcal{L}:=\{\gamma \in \wh{G}: |\wh{1_A}(\gamma)| \geq \sqrt{1-\eta^2/2}\P_G(A)\}.
\end{equation*}
First, note that since $\eta \leq 1$, if $\gamma \in \mathcal{L}$ then
\begin{equation*}
2|\wh{1_A}(\gamma)| \geq \sqrt{2}\P_G(A) > \P_G(S.A) 
\end{equation*}
whence, by Lemma \ref{lem.nxt} (applicable since every $\gamma$ on $G$ is monomial), $\gamma$ is one-dimensional.  It follows that $\mathcal{L}=\LSpec(1_A,\eta)$.
Secondly,
\begin{eqnarray*}
\sum_{ \gamma \not \in \mathcal{L}}{d_\gamma^2|\wh{1_A}(\gamma)|^{2k}} & \leq &
(\sqrt{1-\eta^2/2}\P_G(A))^{2k-2}\sum_{\gamma \in \wh{G}}{d_\gamma^2|\wh{1_A}(\gamma)|^2}\\ & = & (1-\eta^2/2)^{k-1}\P_G(A)^{2k-2}\sum_{\gamma \in \wh{G}}{d_\gamma\|\wh{1_A}(\gamma)\|^2}\\ & = & (1-\eta^2/2)^{k-1}\P_G(A)^{2k-1},
\end{eqnarray*}
by (\ref{eqn.l2}) and Parseval's theorem. Inserting the hypothesis on $k$ we get that
\begin{equation*}
\sum_{ \gamma \not \in \mathcal{L}}{d_\gamma^2|\wh{1_A}(\gamma)|^{2k}}  \leq \frac{\P_G(A)^{2k}}{2\P_G(A^k)}
\end{equation*}
and the lemma then follows from the triangle inequality in (\ref{eqn.polCS2}), and
(\ref{eqn.polynomialgrowthCS}).
\end{proof}
The next proposition concerns the growth of large spectra -- recall that they are subsets of the abelian group $\Lin(G)$.
\begin{proposition}\label{item.LSpecdoubling}
Suppose that $ \epsilon \in (0,1]$ and $d\geq 1$ are parameters and
\begin{equation*}
\P_G(A^k)\leq k^d\P_G(A) \textrm{ whenever } 64d \log 32d \leq k \leq 128 \epsilon^{-2}d\log 32 \epsilon^{-2}d.
\end{equation*}
Then, either $\epsilon^{-1} = O(d \log^22d)$, or there is a set $X \subset \LSpec(A,2\epsilon)$ with size $O(d \log^22\epsilon^{-1}d)$
such that
\begin{equation*}
\LSpec(A,\epsilon) + \LSpec(A,\epsilon) \subset \Span(X) +
\LSpec(A,\epsilon).
\end{equation*}
\end{proposition}
\begin{proof}
By the triangle inequality and the fact that $\gamma(x)$ is unitary we have that
\begin{equation*}
\|\wh{1_A}(\gamma)v\|= \|\E_{x \in G}{1_A(x) \gamma(x)v}\| \leq \E_{x \in G}{1_A(x)\|\gamma(x)v\|} = \P_G(A)
\end{equation*}
whenever $v$ is a unit vector in $\C^{d_\gamma}$.  It follows that $|\wh{1_A}(\gamma)| \leq \P_G(A)$.  This coupled with the fact that all representations in $\LSpec(A,\eta)$ are one-dimensional gives
\begin{equation}\label{eqn.q1}
\sum_{\gamma \in \LSpec(A,\eta)}{d_\gamma^2|\wh{1_A}(\gamma)|^{2k}} \leq
|\LSpec(A,\eta)|\P_G(A)^{2k}.
\end{equation}
Similarly, if $\eta \leq 1/2$ then
\begin{equation}\label{eqn.q2}
\sum_{\gamma \in \wh{G}}{d_\gamma^2|\wh{1_A}(\gamma)|^{2k}} \geq
|\LSpec(A,2\eta)|(\sqrt{1-2\eta^2}\P_G(A))^{2k}.
\end{equation}
Now, write $k_{\eta,d}:=\lceil 16 \eta^{-2}d \log 8\eta^{-2} d \rceil$ so that
\begin{equation*}
(1-\eta^2/2)^{k_{\eta,d}-1} \leq 1/2k_{\eta,d}^d,
\end{equation*}
but since
\begin{equation*}
\P_G(A^k) \leq k^d\P_G(A)
\end{equation*}
whenever $64d \log 32d \leq k \leq 128\epsilon^{-2} d \log 32\epsilon^{-2}d$ we conclude that
\begin{equation*}
(1-\eta^2/2)^{k_{\eta,d}-1} \leq \P_G(A)/2\P_G(A^{k_{\eta,d}})
\end{equation*}
whenever $\eta \in (\epsilon/2,1/2]$.  Now, Lemma \ref{lem.prelim} tells us that
\begin{equation*}
\sum_{\gamma \in \LSpec(A,\eta)}{d_\gamma^2|\wh{1_A}(\gamma)|^{2k_{\eta,d}}} \geq \frac{1}{2}\sum_{\gamma \in \wh{G}}{d_\gamma^2|\wh{1_A}(\gamma)|^{2k_{\eta,d}}}.
\end{equation*}
It follows from this and (\ref{eqn.q1}) and (\ref{eqn.q2}) that
\begin{equation*}
|\LSpec(A,2\eta)| \leq 
2(1-2\eta^2)^{-k_{\eta,d}}|\LSpec(A,\eta)|
\end{equation*}
whenever $\eta \in (\epsilon/2,1/2]$.  On the other hand
\begin{equation*}
2(1-2\eta^2)^{-k_{\eta,d}} \leq 2(1+4\eta^2)^{k_{\eta,d}} \leq 2\exp(4\eta^2k_{\eta,d}) \leq \exp(O(d \log 2\epsilon^{-1}d))
\end{equation*}
for the same range of $\eta$, so
\begin{equation*}
|\LSpec(A,2\eta)| \leq 
\exp(O(d \log 2\epsilon^{-1}d))|\LSpec(A,\eta)|.
\end{equation*}
Repeated application gives that for any integer $r>1$ with $(2r+1/2)\epsilon \leq 1$,
\begin{equation*}
|\LSpec(A,(2r+1/2)\epsilon)| \leq 
\exp(O(d\log r \log 2\epsilon^{-1}d))|\LSpec(A,\epsilon/2)|.
\end{equation*}
It follows that either $\epsilon^{-1}=O(d\log^2 2d)$ or we
may pick $r$ with $r=O(d\log^2 2\epsilon^{-1}d)$ such that $(2r+1/2)\epsilon \leq 1$ and
\begin{equation*}
|\LSpec(A,(2r+1/2)\epsilon)| <
2^r|\LSpec(A,\epsilon/2)|.
\end{equation*}
Thus, since $\LSpec(A,(2r+1/2)\epsilon) \supset r\LSpec(A,2\epsilon)
+ \LSpec(A,\epsilon/2)$, by Chang's covering lemma (specifically as formulated in \cite[Lemma 4.2]{san::8} for example) we have a set $X\subset \LSpec(A,2\epsilon)$ with $|X| \leq r$
such that
\begin{equation*}
\LSpec(A,2\epsilon) \subset \Span(X) + \LSpec(A,\epsilon/2) -
\LSpec(A,\epsilon/2).
\end{equation*}
The result follows.
\end{proof}
The second key result of the section effectively estimates the size of the large spectrum.
\begin{proposition}\label{item.LSpecsize}
Suppose $\epsilon \in (0,1]$ and $d \geq 1$ are parameters, and there is some $k \geq 16\epsilon^{-2}d \log 8\epsilon^{-2} d$ such that
\begin{equation*}
\P_G(A^k)\leq k^d\P_G(A).
\end{equation*}
Then 
\begin{equation*}
\P_G(\Bohr(\LSpec(A,\epsilon),1/2\pi)) \leq 8k^d\P_G(A).
\end{equation*}
\end{proposition}
\begin{proof} 
Write $\beta$ for the probability measure on $\Bohr(\LSpec(A,\epsilon),1/2\pi)$.  Suppose that $\gamma \in \LSpec(A,\epsilon)$, so $\gamma$ is one-dimensional.  Then, for every $x \in \Bohr(\LSpec(A,\epsilon),1/2\pi)$ we have
\begin{equation*}
| - \gamma(x)| = \sqrt{2(1-\cos (\pi \|\gamma(x)\|))} \leq \pi
\|\gamma(x)\| \leq 1/2.
\end{equation*}
Integrating the above calculation with respect to $d\beta$ and applying the triangle inequality tells us that 
\begin{equation*}
|1- \wh{\beta}(\gamma)| \leq \E_{x \in G}{|1 - \gamma(x)|\beta(x)} \leq 1/2.
\end{equation*}
It follows, again by the triangle inequality, that $\|\wh{\beta}(\gamma)\|=|\wh{\beta}(\gamma)| \geq 1/2$.

Now, recalling the definition of $f$ from the start of the section,
\begin{eqnarray*}
\sum_{\gamma \in \wh{G}}{d_\gamma\|\wh{f}(\gamma)\wh{\beta}(\gamma)\|^2} & \geq & 2^{-2}\sum_{\gamma \in \LSpec(A,\epsilon)}{d_\gamma\|\wh{f}(\gamma)\|^2}\\ & = & 2^{-2} \sum_{\gamma \in \LSpec(A,\epsilon)}{d_\gamma^2|\wh{1_A}(\gamma)|^{2k}},
\end{eqnarray*}
since $\wh{f}(\gamma)=\wh{1_A}(\gamma)^k$ and $\gamma$ is one-dimensional.

On the other hand, since $k \geq 16 \epsilon^{-2}d \log 8 \epsilon^{-2}d$ we have that
\begin{equation*}
(1-\epsilon^2/2)^{k-1} \leq 1/2k^d \leq \P_G(A)/2\P_G(A^k).
\end{equation*}
It follows from Lemma \ref{lem.prelim} that 
\begin{equation*}
\sum_{\gamma \in \LSpec(A,\epsilon)}{d_\gamma^2|\wh{1_A}(\gamma)|^{2k}}\geq \frac{\P_G(A)^{2k}}{2\P_G(A^k)},
\end{equation*}
which, when inserted into the previous expression gives
\begin{equation*}
\sum_{\gamma \in \wh{G}}{d_\gamma\|\wh{f}(\gamma)\wh{\beta}(\gamma)\|^2} \geq \frac{\P_G(A)^{2k}}{2^{3}\P_G(A^k)}.
\end{equation*}
However, by Parseval's theorem and Young's inequality we get that
\begin{eqnarray*}
\frac{\P_G(A)^{2k}}{2^3\P_G(A^k)} \leq \|f \ast \beta\|_{L^2(\P_G)}^2 & \leq &\P_G(A)^{2k-2}\|1_A \ast \beta\|_{L^2(\P_G)}^2\\ & \leq &  \P_G(A)^{2k-2}\|1_A \ast
 \beta\|_{L^1(\P_G)} \|1_A \ast \beta\|_{L^\infty(\P_G)}.
\end{eqnarray*}
Since $\|1_{A}\ast \beta\|_{L^1(\P_G)} = \P_G(A)$ we conclude that
\begin{equation*}
\frac{\P_G(A)}{2^3\P_G(A^{k})} \leq \|1_A \ast
\beta\|_{L^\infty(\P_G)} \leq \frac{\P_G(A)}{\P_G(\Bohr(\LSpec(A,\epsilon),1/2\pi))}.
\end{equation*}
The result follows.
\end{proof}

\section{Linear Bohr sets with large spectra as frequency sets}

Finally we turn our attention to combining linear Bohr sets with structured spectra.  The following idea was introduced by Green and Ruzsa in \cite{greruz::0}, although the proof below is a slight adaptation of one appearing in \cite{taovu::}.
\begin{proposition}\label{prop.lowerbound}
Suppose that $G$ is a finite group and $A$ is a finite set, $l$
is a positive integer such that $\P_G(A^l) \leq K\P_G(A^{l-1})$ and
$\epsilon \in (0,1]$ is a parameter. Then
\begin{equation*}
AA^{-1} \subset \Bohr(\LSpec(A^l,\epsilon),2\epsilon\sqrt{2K}).
\end{equation*}
\end{proposition}
\begin{proof}
Write $\delta = 1-\sqrt{1-\epsilon^2/2}$ and suppose that $\gamma
\in \LSpec(A^l,\epsilon)$. Since $\gamma$ is one-dimensional
\begin{equation*}
|\wh{1_{A^l}}(\gamma)|\geq (1-\delta)\P_G(A^l),
\end{equation*}
Thus there is a phase $\omega \in S^1$ such that
\begin{equation*}
\E_{x \in G}{1_{A^l}(x)\omega\gamma(x)} \geq (1-\delta)\P_G(A^l).
\end{equation*}
It follows that
\begin{equation*}
\E_{x \in G}{1_{A^l}(x)|1-\omega\gamma(x)|^2} =2\E_{x \in G}{1_{A^l}(x)(1-\omega\gamma(x))} \leq 2\delta\P_G(A^l).
\end{equation*}
If $y_0,y_1 \in A$ then
\begin{equation*}
\E_{x \in G}{1_{A^{l-1}}(x)|1-\omega\gamma(y_i)\gamma(x)|^2} \leq
\E_{x \in G}{1_{A^l}(x)|1-\omega\gamma(x)|^2} \leq 2\delta\P_G(A^l).
\end{equation*}
The Cauchy-Schwarz inequality tells us that
\begin{equation*}
|1-\gamma(y_0-y_1)|^2 \leq 2(|1-\omega\gamma(y_0)\gamma(x)|^2
+|1-\omega\gamma(y_1)\gamma(x)|^2)
\end{equation*}
for all $x \in G$, whence
\begin{equation*}
\E{1_{A^{l-1}}|1-\gamma(y_0-y_1)|^2} \leq 2^3\delta\P_G(A^l).
\end{equation*}
On the other hand
\begin{equation*}
|1-\gamma(x)|^2 = 2(1-\cos(\pi\|\gamma(x)\|)) \geq
2^{-1}\|\gamma(x)\|^2,
\end{equation*}
from which the result follows.
\end{proof}

\section{The proof of the main theorem}\label{sec.proofoftheorem}

We are now in a position to prove our main theorem.
\begin{proof}[Proof of Theorem \ref{thm.maintheorem}]
First, we may assume that $A$ generates $G$ since $G$ is hereditarily monomial, and $A$ is normal in $\langle A \rangle$.  Now, let $C$ be the constant implicit in the first possible conclusion of Proposition \ref{item.LSpecdoubling}, so that $\epsilon^{-1} \leq C d \log^22d$ in that case.

By the pigeon-hole principle there is some integer $l$ with $l=O(d \log 2d)$ such that $\P_G(A^{l+1}) < \sqrt{2}\P_G(A^{l-1})$, from which it follows that
\begin{equation*}
\P_G(A.A^l) < \sqrt{2}\P_G(A^l) \textrm{ and } \P_G(A^l) < \sqrt{2}\P_G(A^{l-1}).
\end{equation*}
Let $d'=O(d)$ be such that $\P_G((A^l)^n) \leq n^{d'}\P_G(A^l)$ for all $n \geq d'\log2 d'$, and finally let $\epsilon^{-1} := 2^{9}(1+C)d'\log^2 2d'$.

In view of the choice of $\epsilon$, by Proposition \ref{item.LSpecdoubling}
applied to $A^l$ there is some set $X\subset \LSpec(A^l,2\epsilon)$ with $|X|=O(d'\log^2 2\epsilon^{-1}d')=O(d\log^22d)$ such that
\begin{equation*}
\LSpec(A^l,\epsilon)+\LSpec(A^l,\epsilon)\subset \Span(X) +
\LSpec(A^l,\epsilon).
\end{equation*}
Consider the ball $B=\Bohr(\LSpec(A^l,\epsilon)\cup X,2^{-4})$. First, by Proposition \ref{prop.approximateannihilatorsofapproximategroups}, this ball is $O(d\log^32d)$-dimensional. Secondly, since $\LSpec(A^l,\epsilon) \cup X
\subset \LSpec(A^l,2\epsilon)$ we have
\begin{equation*}
AA^{-1} \subset \Bohr(\LSpec(A^l,2\epsilon),8\epsilon)
\subset \Bohr(\LSpec(A^l,2\epsilon),2^{-4})\subset B,
\end{equation*}
by Proposition \ref{prop.lowerbound}. Finally, Proposition
\ref{item.LSpecsize} ensures that
\begin{equation*}
\P_{G}(B) \leq \exp(O(d'\log2 d'))\P_{G}(A^l) \leq \exp(O(d\log2 d))\P_{G}(A).
\end{equation*}
The result is complete.
\end{proof}

\appendix

\section{From small tripling to polynomial growth in nilpotent groups}

The object of this appendix is to show Proposition \ref{prop.done} below, that a subset of a nilpotent group with small tripling has relative polynomial growth.  Theorem \ref{thm.mainnilpotenttheorem} is then an immediate corollary of Theorem \ref{thm.maintheorem} and this proposition.

Our argument couples the covering method of Ruzsa (introduced to the non-abelian setting by Tao in \cite{tao::6}) with the following result due, independently to Bass \cite{bas::} and Guivarc'h \cite{gui::}.
\begin{theorem}
Suppose that $G$ is a finitely generated nilpotent group with lower central series $G=G_0 \rhd G_1 \rhd \dots \rhd G_{d+1}=\{1_G\}$. Then for every set $A \subset G$ there is a constant $C_{G,A}$ dependent only on $G$ and $A$ such that $|A^n| \leq C_{G,A}n^{d(G)}$ for all $n \geq 1$, where
\begin{equation*}
d(G)=\sum_{i =0}^d{i\rk(G_i/G_{i+1})}.
\end{equation*}
\end{theorem}
Recall that if $H$ is an abelian group then $\rk(H)$ is the size of the largest set of torsion free independent elements of $H$.
\begin{corollary}\label{cor.c}
Suppose that $G$ is a finite nilpotent group of class $c$ and $X \subset G$ is a set of size $K$. Then
\begin{equation*}
|X^n| \leq O_{K,c}(n^{O_{K,c}(1)}).
\end{equation*}
\end{corollary}
\begin{proof}
By restricting $G$ to the group generated by $X$ (which is also nilpotent) we may assume that $G$ has at most $K$ generators. Let $S$ be a set of $K$ elements and $F(S)$ be a free nilpotent group based on $S$.  Since $F(S)$ is free any bijection $\phi:S \rightarrow X$ extends to a homomorphism $\tilde{\phi}:F(S) \rightarrow G$, such that $\tilde{\phi}|_S = \phi$. It follows that
\begin{equation*}
|S| = |X| = K \textrm{ and } |X^n| \leq |S^n| \textrm{ for all }n \geq 1.
\end{equation*}
Since $F(S)$ has a set of $K$ generators we see that $d(F(S))=O_{K,c}(1)$, where $d(F(S))$ is as in the Bass-Guivarc'h theorem. Applying this we conclude that
\begin{equation*}
|X^n| \leq |S^n| \leq C_{F(S),S}n^{O_{K,c}(1)}=O_{K,c}(n^{O_{K,c}(1)})
\end{equation*}
to get the result.
\end{proof}
Note that the argument above is completely ineffective; no bound results so we might as well have used the earlier, weaker version of the Bass-Guivarc'h theorem due to Wolf \cite{wol::1}. It seems likely that a direct modification of either the Bass-Guivarc'h theorem or Wolf's theorem could lead to an effective bound and be used to establish Corollary \ref{cor.c} directly.
\begin{proposition}\label{prop.done}
Suppose that $G$ is a finite nilpotent group of class $c$ and $A \subset G$ has $|A^3| \leq K|A|$. Then
\begin{equation*}
|A^n| = O_{K,c}(n^{O_{K,c}(1)}|A|).
\end{equation*}
\end{proposition}
\begin{proof}
By, for example, \cite[Lemma 3.4]{tao::6} we have that $|AA^{-1}AA^{-1}| \leq K^{O(1)}|A|$ and $|AA^{-1}| \leq K^{O(1)}|A|$.

Let $X \subset AA^{-1}AA^{-1}$ be a maximal $A$-separated set. As usual it follows that $|X||A| \leq |AA^{-1}AA^{-1}|$, whence $|X| \leq K^{O(1)}$ and $AA^{-1}AA^{-1} \subset XAA^{-1}$. Thus 
\begin{equation*}
(AA^{-1})^n \subset X^{n-1}AA^{-1}.
\end{equation*}
Now, since $|X| \leq K^{O(1)}$ by Corollary \ref{cor.c} we have that
\begin{equation*}
|X^{n-1}| = O_{K,c}(n^{O_{K,c}(1)}),
\end{equation*}
and the result follows since $|AA^{-1}| \leq K^{O(1)}|A|$ and $|A^n| \leq |(AA^{-1})^n|$.
\end{proof}

\section*{Acknowledgements}

The author would like to thank Emmanuel Breuillard, Ben Green, Terence Tao and Matt Tointon for useful comments, and an anonymous referee for careful scrutiny of the paper.

\bibliographystyle{halpha}

\bibliography{references}

\newcommand{\etalchar}[1]{$^{#1}$}
\begin{thebibliography}{BDJ{\etalchar{+}}82}

\bibitem[Bas72]{bas::}
H.~Bass.
\newblock The degree of polynomial growth of finitely generated nilpotent
  groups.
\newblock {\em Proc. London Math. Soc. (3)}, 25:603--614, 1972.

\bibitem[BDJ{\etalchar{+}}82]{bradesjoh::}
H.~G. Bray, W.~E. Deskins, D.~Johnson, J.~F. Humphreys, B.~M. Puttaswamaiah,
  P.~Venzke, and G.~L. Walls.
\newblock {\em Between nilpotent and solvable}.
\newblock Polygonal Publ. House, Washington, N. J., 1982.
\newblock Edited and with a preface by Michael Weinstein.

\bibitem[BG11]{bregre::}
E.~Breuillard and B.~J. Green.
\newblock Approximate groups. {I}: the torsion-free nilpotent case.
\newblock {\em J. Inst. Math. Jussieu}, 10(1):37--57, 2011.

\bibitem[Bil99]{bil::}
Y.~Bilu.
\newblock Structure of sets with small sumset.
\newblock {\em Ast\'erisque}, (258):xi, 77--108, 1999.
\newblock Structure theory of set addition.

\bibitem[BKT04]{boukattao::}
J.~Bourgain, N.~H. Katz, and T.~C. Tao.
\newblock A sum-product estimate in finite fields, and applications.
\newblock {\em Geom. Funct. Anal.}, 14(1):27--57, 2004.

\bibitem[Bou99]{bou::5}
J.~Bourgain.
\newblock On triples in arithmetic progression.
\newblock {\em Geom. Funct. Anal.}, 9(5):968--984, 1999.

\bibitem[Cha02]{cha::0}
M.-C. Chang.
\newblock A polynomial bound in {F}re{\u\i}man's theorem.
\newblock {\em Duke Math. J.}, 113(3):399--419, 2002.

\bibitem[Cha07]{cha::2}
M.-C. Chang.
\newblock Additive and multiplicative structure in matrix spaces.
\newblock {\em Combin. Probab. Comput.}, 16(2):219--238, 2007.

\bibitem[Cha08]{cha::3}
M.-C. Chang.
\newblock Product theorems in {${\rm SL}\sb 2$} and {${\rm SL}\sb 3$}.
\newblock {\em J. Inst. Math. Jussieu}, 7(1):1--25, 2008.

\bibitem[DSV03]{davsarval::}
G.~Davidoff, P.~Sarnak, and A.~Valette.
\newblock {\em Elementary number theory, group theory, and {R}amanujan graphs},
  volume~55 of {\em London Mathematical Society Student Texts}.
\newblock Cambridge University Press, Cambridge, 2003.

\bibitem[EK01]{elekir::}
G.~Elekes and Z.~Kir{\'a}ly.
\newblock On the combinatorics of projective mappings.
\newblock {\em J. Algebraic Combin.}, 14(3):183--197, 2001.

\bibitem[FKP10]{fiskatpen::}
D.~Fisher, N.~H. Katz, and I.~Peng.
\newblock Approximate multiplicative groups in nilpotent {L}ie groups.
\newblock {\em Proc. Amer. Math. Soc.}, 138(5):1575--1580, 2010.

\bibitem[Fre73]{fre::0}
G.~A. Fre{\u\i}man.
\newblock {\em Foundations of a structural theory of set addition}.
\newblock American Mathematical Society, Providence, R. I., 1973.
\newblock Translated from the Russian, Translations of Mathematical Monographs,
  Vol 37.

\bibitem[Gow98]{gow::4}
W.~T. Gowers.
\newblock A new proof of {S}zemer\'edi's theorem for arithmetic progressions of
  length four.
\newblock {\em Geom. Funct. Anal.}, 8(3):529--551, 1998.

\bibitem[Gow08]{gow::2}
W.~T. Gowers.
\newblock Quasirandom groups.
\newblock {\em Comb. Probab. Comput.}, 17(3):363--387, 2008.

\bibitem[GR07]{greruz::0}
B.~J. Green and I.~Z. Ruzsa.
\newblock Fre{\u\i}man's theorem in an arbitrary abelian group.
\newblock {\em J. Lond. Math. Soc. (2)}, 75(1):163--175, 2007.

\bibitem[Gre05]{gre::02}
B.~J. Green.
\newblock A {S}zemer\'edi-type regularity lemma in abelian groups, with
  applications.
\newblock {\em Geom. Funct. Anal.}, 15(2):340--376, 2005.

\bibitem[Gro81]{gro::}
M.~Gromov.
\newblock Groups of polynomial growth and expanding maps.
\newblock {\em Inst. Hautes \'Etudes Sci. Publ. Math.}, (53):53--73, 1981.

\bibitem[GT08]{gretao::1}
B.~J. Green and T.~C. Tao.
\newblock An inverse theorem for the {G}owers {$U\sp 3(G)$} norm.
\newblock {\em Proc. Edinb. Math. Soc. (2)}, 51(1):73--153, 2008.

\bibitem[Gui71]{gui::}
Y.~Guivarc'h.
\newblock Groupes de {L}ie \`a croissance polynomiale.
\newblock {\em C. R. Acad. Sci. Paris S\'er. A-B}, 272:A1695--A1696, 1971.

\bibitem[Hel08]{hel::}
H.~A. Helfgott.
\newblock Growth and generation in ${SL}_2(\mathbb{Z}/p\mathbb{Z})$.
\newblock {\em Ann. of Math. (2)}, 167:601--623, 2008.

\bibitem[Hel11]{hel::0}
H.~A. Helfgott.
\newblock Growth in ${SL}_3(\mathbb{Z}/p\mathbb{Z})$.
\newblock {\em J. Eur. Math. Soc.}, 13(3):761--851, 2011.

\bibitem[HLS98]{hamllaser::}
Y.~O. Hamidoune, A.~S. Llad{\'o}, and O.~Serra.
\newblock On subsets with small product in torsion-free groups.
\newblock {\em Combinatorica}, 18(4):529--540, 1998.

\bibitem[Hup98]{hup::}
B.~Huppert.
\newblock {\em Character theory of finite groups}, volume~25 of {\em de Gruyter
  Expositions in Mathematics}.
\newblock Walter de Gruyter \& Co., Berlin, 1998.

\bibitem[Isa94]{isa::}
I.~M. Isaacs.
\newblock {\em Character theory of finite groups}.
\newblock Dover Publications Inc., New York, 1994.
\newblock Corrected reprint of the 1976 original [Academic Press, New York;
  MR0460423 (57 \#417)].

\bibitem[Kir94]{kir::}
A.~A. Kirillov, editor.
\newblock {\em Representation theory and noncommutative harmonic analysis.
  {I}}, volume~22 of {\em Encyclopaedia of Mathematical Sciences}.
\newblock Springer-Verlag, Berlin, 1994.
\newblock Fundamental concepts. Representations of Virasoro and affine
  algebras, A translation of {\it Current problems in mathematics. Fundamental
  directions. Vol.\ 22} (Russian), Akad.\ Nauk SSSR, Vsesoyuz.\ Inst.\ Nauchn.\
  i Tekhn.\ Inform., Moscow, 1988 [ MR0942946 (88k:22001)], Translation by V.
  Sou\v cek.

\bibitem[Lin01]{lin::}
E.~Lindenstrauss.
\newblock Pointwise theorems for amenable groups.
\newblock {\em Invent. Math.}, 146(2):259--295, 2001.

\bibitem[LM08]{leemak::}
J.~R. Lee and Y.~Makarychev.
\newblock Eigenvalue multiplicity and volume growth.
\newblock {\em Journal of Topology and Analysis}, 2008, arXiv:0806.1745.
\newblock to appear.

\bibitem[LPS88]{lubphisar::}
A.~Lubotzky, R.~Phillips, and P.~Sarnak.
\newblock Ramanujan graphs.
\newblock {\em Combinatorica}, 8(3):261--277, 1988.

\bibitem[N{\v{S}}82]{nauste::}
M.~A. Na{\u\i}mark and A.~I. {\v{S}}tern.
\newblock {\em Theory of group representations}, volume 246 of {\em Grundlehren
  der Mathematischen Wissenschaften [Fundamental Principles of Mathematical
  Sciences]}.
\newblock Springer-Verlag, New York, 1982.
\newblock Translated from the Russian by Elizabeth Hewitt, Translation edited
  by Edwin Hewitt.

\bibitem[Ruz94]{ruz::9}
I.~Z. Ruzsa.
\newblock Generalized arithmetical progressions and sumsets.
\newblock {\em Acta Math. Hungar.}, 65(4):379--388, 1994.

\bibitem[Ruz99]{ruz::01}
I.~Z. Ruzsa.
\newblock An analog of {F}re{\u\i}man's theorem in groups.
\newblock {\em Ast\'erisque}, (258):xv, 323--326, 1999.
\newblock Structure theory of set addition.

\bibitem[San09]{san::8}
T.~Sanders.
\newblock A {F}re\u\i man-type theorem for locally compact abelian groups.
\newblock {\em Ann. Inst. Fourier (Grenoble)}, 59(4):1321--1335, 2009,
  arXiv:0710.2545.

\bibitem[Sch03]{sch::0}
T.~Schoen.
\newblock Multiple set addition in {$\mathbb{Z}\sb p$}.
\newblock {\em Integers}, 3:A17, 6 pp. (electronic), 2003.

\bibitem[Shk06a]{shk::0}
I.~D. Shkredov.
\newblock On a generalization of {S}zemer\'edi's theorem.
\newblock {\em Proc. London Math. Soc. (3)}, 93(3):723--760, 2006.

\bibitem[Shk06b]{shk::1}
I.~D. Shkredov.
\newblock On a problem of {G}owers.
\newblock {\em Izv. Ross. Akad. Nauk Ser. Mat.}, 70(2):179--221, 2006.

\bibitem[SSV05]{sudszevu::}
B.~Sudakov, E.~Szemer{\'e}di, and V.~H. Vu.
\newblock On a question of {E}rd{\H o}s and {M}oser.
\newblock {\em Duke Math. J.}, 129(1):129--155, 2005.

\bibitem[ST09]{shatao::}
Y.~Shalom and T.~C. Tao.
\newblock A finitary version of {G}romov's polynomial growth theorem.
\newblock 2009, arXiv:0910.4148.

\bibitem[SV05]{szevu::}
E.~Szemer{\'e}di and V.~H. Vu.
\newblock Long arithmetic progressions in sum-sets and the number of
  {$x$}-sum-free sets.
\newblock {\em Proc. London Math. Soc. (3)}, 90(2):273--296, 2005.

\bibitem[Tao05]{tao::2}
T.~C. Tao.
\newblock Fourier analysis on finite non-abelian groups.
\newblock Available at {\tt \verb!www.math.ucla.edu/~tao!}, 2005.

\bibitem[Tao08a]{tao::7}
T.~C. Tao.
\newblock The correspondence principle and finitary ergodic theory.
\newblock {\tt \verb!terrytao.wordpress.com!}, 2008.

\bibitem[Tao08b]{tao::6}
T.~C. Tao.
\newblock Product set estimates for non-commutative groups.
\newblock {\em Combinatorica}, 28(5):547--594, 2008.

\bibitem[Tao10]{tao::9}
T.~C. Tao.
\newblock Fre{\u\i}man's theorem for solvable groups.
\newblock {\em Contrib. Disc. Math.}, 5(2):137--184, 2010.

\bibitem[TV06]{taovu::}
T.~C. Tao and H.~V. Vu.
\newblock {\em Additive combinatorics}, volume 105 of {\em Cambridge Studies in
  Advanced Mathematics}.
\newblock Cambridge University Press, Cambridge, 2006.

\bibitem[TV07]{taovu::0}
T.~C. Tao and V.~H. Vu.
\newblock On the singularity probability of random {B}ernoulli matrices.
\newblock {\em J. Amer. Math. Soc.}, 20(3):603--628 (electronic), 2007.

\bibitem[Wol68]{wol::1}
J.~A. Wolf.
\newblock Growth of finitely generated solvable groups and curvature of
  {R}iemanniann manifolds.
\newblock {\em J. Differential Geometry}, 2:421--446, 1968.

\end{thebibliography}

\end{document}